\documentclass{amsart}
\usepackage{latexsym,amsxtra,amscd,ifthen}
\usepackage{amsfonts}
\usepackage{color,graphicx}
\usepackage{verbatim}
\usepackage{amsmath}
\usepackage{amsthm}
\usepackage{amssymb}
\usepackage{url}

\theoremstyle{plain}

\newtheorem{theorem}{Theorem}
\newtheorem{lemma}[theorem]{Lemma}
\newtheorem{proposition}[theorem]{Proposition}
\newtheorem{corollary}[theorem]{Corollary}
\newtheorem{conjecture}[theorem]{Conjecture}

\numberwithin{theorem}{section}
\numberwithin{equation}{theorem}

\theoremstyle{definition}
\newtheorem{definition}[theorem]{Definition}
\newtheorem{example}[theorem]{Example}
\newtheorem{remark}[theorem]{Remark}
\newtheorem{question}[theorem]{Question}
\newtheorem*{question*}{Question}

\newcommand{\fm}{\mathfrak{m}}
\newcommand{\bfl}{\mathbf{l}}

\DeclareMathOperator{\End}{End}
\DeclareMathOperator{\Ext}{Ext}
\DeclareMathOperator{\Hom}{Hom}

\DeclareMathOperator{\GKdim}{GKdim}
\DeclareMathOperator{\gldim}{gldim}
\DeclareMathOperator{\MaxSpec}{MaxSpec}
\DeclareMathOperator{\Spec}{Spec}
\DeclareMathOperator{\Kdim}{Kdim}

\begin{document}

\title{Cancellation problem for \\
AS-regular algebras of dimension three}

\author{X. Tang, H. J. Venegas Ram{\' i}rez and J.J. Zhang}

\address{Tang: Department of Mathematics \& Computer Science, 
Fayetteville State University, Fayetteville, NC 28301,
USA}

\email{xtang@uncfsu.edu}

\address{Venegas Ram{\' i}rez: Departamento de Matem{\'a}ticas,
Universidad Nacional de Colombia, Sede Bogot{\'a}, Colombia}

\email{hjvenegasr@unal.edu.co}

\address{Zhang: Department of Mathematics, Box 354350,
University of Washington, Seattle, Washington 98195, USA}

\email{zhang@math.washington.edu}

\begin{abstract}
We study a noncommutative version of the Zariski cancellation problem
for some classes of connected graded Artin-Schelter regular algebras
of global dimension three.
\end{abstract}

\subjclass[2010]{Primary 16P99, 16W99}

%16P99 View Publications (1991-now) None of the above, but in this section
%16W99 View Publications (1991-now) None of the above, but in this section
% Rings and algebras with additional structure
%16W20 (1991-now) Automorphisms and endomorphisms
%16E65 View Publications (2000-now) Homological
%conditions on rings (generalizations of regular,
%Gorenstein, Cohen-Macaulay rings, etc.)

\keywords{Zariski cancellation problem, Morita cancellation problem, 
Artin-Schelter regular algebra, finite global dimension}

%\thanks{ }

\maketitle

%\tableofcontents

% \setcounter{section}{-1}
\section*{Introduction}
\label{xxsec0}

The classical Zariski cancellation problem for commutative 
polynomial rings has a long history, see a very nice survey 
paper of Gupta \cite{Gu3} written in 2015. A noncommutative 
version of the Zariski cancellation problem was investigated 
as early as 1970s, see papers by Coleman-Enochs \cite{CE} 
and Brewer-Rutter \cite{BR}, and was re-introduced by Bell 
and the third-named author in 2017 in \cite{BZ1}. During the 
past few years several research groups have been making 
significant contributions to this topic, see for example,
\cite{BHHV, BZ1, BZ2, BY, CPWZ1, CPWZ2, CYZ1, CYZ2, Ga1, GKM, GWY, 
LY, LeWZ, LuWZ, LMZ, NTY, Ta1, Ta2, TZZ, WZ}. Very recently, the
Zariski cancellation problem was introduced for commutative 
Poisson algebras by Gaddis-Wang \cite{GW}.

We are following the terminology introduced in \cite{Gu3,BZ1}.
Recall that an algebra $A$ is called {\it cancellative} if 
any algebra isomorphism 
$$A[t] \cong B[t]$$ 
of polynomial extensions for some algebra $B$ implies that 
$$A\cong B.$$
The famous Zariski Cancellation Problem (abbreviated as ZCP) 
asks if 

\medskip

{\it the commutative polynomial ring $\Bbbk [x_1,\dots,x_n]$ 
over a field $\Bbbk$ is cancellative}

\medskip

\noindent 
for $n\geq 1$, see \cite{Kr, Gu3, BZ1}. It is well-known 
that $\Bbbk[x_1]$ is cancellative by a result of 
Abhyankar-Eakin-Heinzer in 1972 \cite{AEH}. For $n=2$,
$\Bbbk[x_1,x_2]$ is cancellative by a result of Fujita in 
1979 \cite{Fu} and Miyanishi-Sugie in 1980 \cite{MS} in 
characteristic zero and by a result of Russell in 1981 
\cite{Ru} in positive characteristic. The ZCP for $n\geq 3$ 
has been open for many years. A major breakthrough in this 
research area is a remarkable result of Gupta in 2014 
\cite{Gu1, Gu2} which settled the ZCP negatively in positive 
characteristic for $n\geq 3$. {\it The ZCP in characteristic 
zero remains open for $n\geq 3$.} Examples of non-cancellative 
algebras were given by Hochster \cite{Ho}, Danielewski \cite{Da}
and Gupta \cite{Gu1, Gu2}, and can be found in 
\cite[Example 2.5]{LuWZ}.

Our main goal is to study the ZCP for noncommutative noetherian 
connected graded Artin-Schelter regular algebras [Definition 
\ref{xxdef0.1}] of global dimension three. 

\begin{definition}\cite[p.171]{AS}
\label{xxdef0.1}
A connected graded algebra $A$ is called 
{\it Artin-Schelter Gorenstein} (or {\it AS Gorenstein}, 
for short) if the following conditions hold:
\begin{enumerate}
\item[(a)]
$A$ has injective dimension $d<\infty$ on
the left and on the right,
\item[(b)]
$\Ext^i_A({}_A\Bbbk, {}_{A}A)=\Ext^i_{A}(\Bbbk_A,A_A)=0$ for all
$i\neq d$, and
\item[(c)]
$\Ext^d_A({}_A\Bbbk, {}_{A}A)\cong 
\Ext^d_{A}(\Bbbk_A,A_A)\cong \Bbbk(\bfl)$ for some
integer $\bfl$. %Here $\bfl$ is called the {\it AS index} of $A$.
\end{enumerate}
%In this case, we say $A$ is of type $(d,\bfl)$. 
If in addition,
\begin{enumerate}
\item[(d)]
$A$ has finite global dimension, and
\item[(e)]
$A$ has finite Gelfand--Kirillov dimension,
\end{enumerate}
then $A$ is called {\it Artin--Schelter regular} (or {\it AS
regular}, for short) of dimension $d$.
\end{definition}

AS-regular algebras are considered as a noncommutative analogue 
of the commutative polynomial rings. We refer to \cite{Le, LMZ, St} 
for related concepts such as the Auslander regularity and the 
Cohen-Macaulay property which will be used later in Proposition 
\ref{xxpro0.6}. It is well-known that the only AS-regular algebra 
of global dimension one is the polynomial ring $\Bbbk[x_1]$.
%, which is cancellative by a classical result \cite[Corollary 3.4]{AEH}. 
Combining classical results in \cite{Fu, MS} with \cite[Theorem 0.5]{BZ1}, 
every AS-regular algebra of global dimension two (over a base 
field of characteristic zero) is cancellative. On the other hand, 
by the results of Gupta in \cite{Gu1, Gu2}, not every AS-regular 
algebra of global dimension three (or higher) is cancellative. 
Therefore it is natural and sensible to ask which AS-regular 
algebras of global dimension three (or higher) are cancellative. 
In \cite[Corollary 0.9]{LMZ}, the authors showed that several 
classes of AS-regular algebras of global dimension three are 
cancellative. We say $A$ is {\it PI} if it satisfies a 
polynomial identity. Our first result is

\begin{theorem}
\label{xxthm0.2}
Suppose ${\text{char}}\; \Bbbk=0$. Let $A$ be a noetherian 
connected graded AS-regular algebra of global dimension three 
that is generated in degree $1$. If $A$ is not PI, then it is 
cancellative.
\end{theorem}

Theorem \ref{xxthm0.2} covers \cite[Corollary 0.9]{LMZ}. 
For an algebra $A$, let $Z(A)$ denote the center of $A$. Let 
$\GKdim A$ (respectively, $\gldim A$) be the Gelfand-Kirillov 
dimension (respectively, the global dimension) of $A$. For 
AS-regular algebras of higher global dimension, we have the 
following.

\begin{theorem}
\label{xxthm0.3}
Suppose ${\text{char}}\; \Bbbk=0$. Let $A$ be a noetherian 
connected graded domain of finite global dimension that is 
generated in degree $1$. Suppose that 
\begin{enumerate}
\item[(a)]
$\GKdim Z(A)\leq 1$, and 
\item[(b)]
$\gldim A/(t)=\infty$ for every homogeneous central element 
$t$ in $Z(A)$ of positive degree.
\end{enumerate}
Then $A$ is cancellative.
\end{theorem}

We will also show a result similar to Theorem \ref{xxthm0.3} 
for graded isolated singularities which have infinite global 
dimension, see Theorem \ref{xxthm4.3}. For a general noncommutative 
algebra we have the following conjecture, which extends both 
Theorems \ref{xxthm0.2} and \ref{xxthm0.3}.

\begin{conjecture}
\label{xxcon0.4}
Suppose ${\text{char}}\; \Bbbk=0$.
Let $A$ be a noetherian finitely generated prime algebra.
\begin{enumerate}
\item[(1)]
If $\GKdim Z(A)\leq 1$, then $A$ is cancellative.
\item[(2)]
If $\GKdim A=3$ and $A$ is not PI, then $A$ is cancellative.
\end{enumerate}
\end{conjecture}

Originally this Conjecture was made without the hypothesis 
that ${\text{char}}\; \Bbbk=0$. When ${\text{char}}\; \Bbbk>0$,
a counterexample to part (1) was given in a recent paper 
\cite[Theorem 1.1(b)]{BHHV}. See \cite[Question 4.1]{BHHV} 
for a modified and weaker version of Conjecture \ref{xxcon0.4}(1) 
in positive characteristic. 

The cancellation property of Veronese subrings of skew 
polynomial rings was considered in \cite{CYZ2}. 
We have the following improvement of Theorem \ref{xxthm0.2}
concerning the Veronese subrings which provides some 
evidence for Conjecture \ref{xxcon0.4}.

\begin{corollary}
\label{xxcor0.5}
Suppose ${\text{char}}\; \Bbbk=0$. Let $A$ be a noetherian 
connected graded AS-regular algebra of global dimension three 
that is generated in degree $1$. If $A$ is not PI, then the
$d$th Veronese subring $A^{(d)}$ of $A$ is cancellative for 
every $d\geq 1$.
\end{corollary}

The proofs of Theorems \ref{xxthm0.2} and \ref{xxthm0.3} are
related to the following result that establishes that the center of 
the algebras in these two theorems is either $\Bbbk$ or $\Bbbk[t]$. 

\begin{proposition}
\label{xxpro0.6}
Let $A$ be a noetherian connected graded Auslander regular 
Cohen-Macaulay algebra. If $\GKdim Z(A)\leq 1$, then $Z(A)$ 
is either $\Bbbk$ or $\Bbbk[t]$.
\end{proposition}

When $Z(A)=\Bbbk[t]$, we can use Theorem \ref{xxthm1.5} 
which was proved in \cite{LuWZ}. And we have a question 
along this line.

\begin{question}
\label{xxque0.7}
Let $A$ be a noetherian connected graded Auslander regular 
algebra. If $\GKdim Z(A)=2$, what can we say about the 
center $Z(A)$? For example, is $Z(A)$ always noetherian in
this case?
\end{question}

As in the commutative case, it is usually difficult to determine
whether or not an AS-regular algebra is cancellative. For example, 
we are unable to answer the following question.

\begin{question}
\label{xxque0.8}
Let $q\in \Bbbk\setminus\{0,1\}$ be a root of unity. Is the
skew polynomial ring of three variables $\Bbbk_q[x_1,x_2,x_3]$
(or of odd number variables $\Bbbk_q[x_1,x_2,\cdots,x_{2n+3}]$)
cancellative?
\end{question}

When $q=1$ and ${\text{char}}\; \Bbbk=0$, the above question 
is the classical ZCP which has been open for many years 
\cite{Gu3}. Note that if $q=1$ and ${\text{char}}\; \Bbbk>0$,
then $\Bbbk[x_1,\cdots,x_n]$ for $n\geq 3$ is not cancellative
by \cite{Gu1,Gu2}. Question \ref{xxque0.8} is a special case of 
\cite[Question 1.5]{CYZ2} which was stated for a larger class 
of rings, namely, for Veronese subrings of the skew polynomial 
rings. Surprisingly, if $q\in \Bbbk\setminus\{0,1\}$, then the 
skew polynomial ring of even number variables 
$\Bbbk_{q}[x_1,\cdots,x_{2n}]$ is cancellative by
\cite[Theorem 0.8(a)]{BZ1}. 

Several new methods were introduced to deal with the noncommutative
version of the ZCP. For example, methods of discriminants and 
Makar-Limanov invariants were introduced and used in \cite{BZ1}. 
In \cite{LeWZ}, the retractability and detectability were introduced 
to relate the cancellation property. In \cite{LMZ}, Nakayama 
automorphisms were used to show some classes of algebras are 
cancellative. In \cite{LuWZ}, Azumaya locus and 
${\mathcal P}$-discriminant methods were introduced to study the 
cancellation property. One should continue to look for new 
invariants and methods to handle the algebras given in 
Question \ref{xxque0.8}.

The paper is organized as follows. Section 1 contains definitions
and preliminaries. In Section 2, we prove some preliminary results 
necessary for the last section. In Section 3, we present some 
lemmas related to cancellation problem. As an application we 
establish that the universal enveloping algebra of any 3-dimensional 
non-abelian Lie algebra is cancellative [Example \ref{xxex3.10}]. 
Theorems \ref{xxthm0.2}, \ref{xxthm0.3} and Corollary \ref{xxcor0.5} 
are proven in Section 4.

\section{Definitions and Preliminaries}
\label{xxsec1}

We recall some definitions from \cite{BZ1, LeWZ, LuWZ}. Let
$\Bbbk$ be a base field that is algebraically closed. Objects 
in this paper are $\Bbbk$-linear. 

\begin{definition}\cite[Definition 1.1]{BZ1}
\label{xxdef1.1}
Let $A$ be an algebra.
\begin{enumerate}
\item[(1)]
We call $A$ {\it cancellative} if any algebra isomorphism 
$A[t] \cong B[s]$ for some algebra $B$ implies that $A\cong B$.
\item[(2)]
We call $A$ {\it strongly cancellative} if, for each $n \geq 1$, 
any algebra isomorphism
$$A[t_1, \cdots, t_n]\cong B[s_1,\cdots,s_n]$$
for some algebra $B$ implies that $A \cong B$.
\end{enumerate}
\end{definition}

For any algebra $A$, let $M(A)$ denote the category of 
right $A$-modules. 

\begin{definition}\cite[Definition 2.2]{LuWZ}
\label{xxdef1.2}
Let $A$ be an algebra.
\begin{enumerate}
\item[(1)]
We call $A$ {\it m-cancellative} if any equivalence of 
abelian categories $M(A[t]) \cong M(B[s])$ for some 
algebra $B$ implies that $M(A)\cong M(B)$.
\item[(2)]
We call $A$ {\it strongly m-cancellative} if, for each 
$n \geq 1$, any equivalence of abelian categories
$$M(A[t_1, \cdots, t_n])\cong M(B[s_1,\cdots,s_n])$$
for some algebra $B$ implies that $M(A) \cong M(B)$. 
\end{enumerate}
The letter $m$ here stands for the word ``Morita''.
\end{definition}

This Morita version of the cancellation property is a
natural generalization of the original Zariski cancellation 
property when we study noncommutative algebras. 

Let $Z$ be a commutative ring over the base field $\Bbbk$, which 
is usually the center of a noncommutative algebra. We now recall 
the definition of ${\mathcal P}$-discriminant for a property 
${\mathcal P}$. Let $\Spec Z$ denote the prime spectrum of $Z$ 
and 
$$\MaxSpec(Z):=\{\fm \mid \fm {\text{ is a maximal ideal of }} Z\}$$ 
is the maximal spectrum of $Z$. For any $S \subseteq \Spec Z$, 
$I(S)$ is the ideal of $Z$ vanishing on $S$, namely,
$$I(S)=\bigcap_{{\mathfrak p}\in S} {\mathfrak p}.$$ 

For any algebra $A$, $A^{\times}$ denotes the set of invertible 
elements in $A$. A property ${\mathcal P}$ considered in the 
following means a property defined on a class of algebras that 
is an invariant under algebra isomorphisms. 

\begin{definition} \cite[Definition 3.3]{LuWZ}
\label{xxdef1.3} 
Let $A$ be an algebra, $Z:=Z(A)$ be the center of $A$. 
Let ${\mathcal P}$ be a property defined for $\Bbbk$-algebras
{\rm{(}}not necessarily a Morita invariant{\rm{)}}. 
\begin{enumerate}
\item[(1)]
The {\it ${\mathcal P}$-locus} of $A$ is defined to be
$$L_{\mathcal P}(A):=\{ \fm \in \MaxSpec(Z)\mid A/\fm A {\text{ has 
the property }} {\mathcal P}\}.$$
\item[(2)]
The {\it ${\mathcal P}$-discriminant set} of $A$ is defined to be
$$D_{\mathcal P}(A):=\MaxSpec(Z)\setminus L_{\mathcal P}(A).$$
\item[(3)]
The {\it ${\mathcal P}$-discriminant ideal} of $A$ is defined to be
$$I_{\mathcal P}(A):=I(D_{\mathcal P}(A))\subseteq Z.$$
\item[(4)]
If $I_{\mathcal P}(A)$ is a principal ideal of $Z$ generated by $d\in Z$,
then $d$ is called the {\it ${\mathcal P}$-discriminant} of $A$, denoted 
by $d_{\mathcal P}(A)$. In this case $d_{\mathcal P}(A)$ is unique 
up to an element in $Z^{\times}$.
\item[(5)]
Let ${\mathcal C}$ be a class of algebras over $\Bbbk$. We say that 
${\mathcal P}$ is {\it ${\mathcal C}$-stable} if for every algebra $A$ 
in ${\mathcal C}$ and every $n\geq 1$,
$$I_{\mathcal P}(A\otimes \Bbbk[t_1,\cdots,t_n])=I_{\mathcal P}(A)
\otimes \Bbbk[t_1,\cdots, t_n]$$
as an ideal of $Z\otimes \Bbbk[t_1,\cdots,t_n]$. If ${\mathcal C}$ is 
a singleton $\{A\}$, we simply call ${\mathcal P}$ {\it $A$-stable}.
If ${\mathcal C}$ is the whole collection of $\Bbbk$-algebras with the
center affine over $\Bbbk$, we simply call ${\mathcal P}$ {\it stable}.
\end{enumerate}
\end{definition}

In general, neither $L_{\mathcal P}(A)$ nor $D_{\mathcal P}(A)$ 
is a subscheme of $\Spec Z(A)$.

In this paper we will use another property that is closely 
related to the m-cancellative property. 

Recall from the Morita theory that if $A':=A[t_1,\cdots,t_n]$ 
is Morita equivalent to $B':=B[s_1,\cdots,s_n]$, then
there is an $(A',B')$-bimodule $\Omega$ that is invertible and 
induces naturally algebra isomorphisms $A' \cong \End(\Omega_{B'})$ 
and $ (B')^{op} \cong \End(_{A'}\Omega)$ such that 
$$Z(A')\cong \Hom_{(A',B')}(\Omega, \Omega)\cong Z(B').$$ 
The above isomorphism is denoted by
\begin{equation}
\label{E1.3.1}\tag{E1.3.1}
\omega: Z(A')\longrightarrow Z(B').
\end{equation}

The retractable property was introduced in 
\cite[Definitions 2.1 and 2.5]{LeWZ} and the Morita 
$Z$-retractability was introduced in 
\cite[Definition 3.6]{LuWZ}.

\begin{definition}
\label{xxdef1.4}
Let $A$ be an algebra.
\begin{enumerate}
\item[(1)] %\cite[Definition 2.5(1)]{LeWZ}
We call $A$ {\it $Z$-retractable} 
if, for any algebra $B$, any algebra isomorphism $\phi: A[t] 
\cong B[s]$ implies that $\phi(Z(A))=Z(B)$. If further $\phi(A)=B$, we just say $A$ is {\it retractable}.
\item[(2)] %\cite[Definition 2.5(2)]{LeWZ}
We call $A$ {\it strongly $Z$-retractable} if, for any algebra $B$ 
and integer $n\geq 1$, any algebra isomorphism 
$\phi: A[t_1,\dots, t_n] \cong B[s_1,\dots, s_n]$ 
implies that $\phi(Z(A))=Z(B)$. If further $\phi(A)=B$, we just say $A$ is {\it strongly retractable}.
\item[(3)] %\cite[Definition 2.6(3)]{LuWZ}
We call $A$ {\it m-$Z$-retractable} if, for any algebra $B$, any
equivalence of categories $M(A[t])\cong M(B[s])$ implies that 
$\omega(Z(A))=Z(B)$ where $\omega: Z(A)[t]\to Z(B)[s]$ is given as in 
\eqref{E1.3.1}.
\item[(4)] %\cite[Definition 2.6(4)]{LuWZ}
We call $A$ {\it strongly m-$Z$-retractable} if, for any algebra $B$ and 
any $n\geq 1$, any equivalence of categories 
$M(A[t_1,\cdots, t_n])\cong M(B[s_1,\cdots,s_n])$ 
implies that $\omega(Z(A))=Z(B)$ where 
$$\omega: Z(A)[t_1,\cdots,t_n]\to 
Z(B)[s_1,\cdots,s_n]$$ is given as in \eqref{E1.3.1}.
\end{enumerate}
\end{definition}

The following theorem was proved in \cite[Corollary 3.11 and 
Lemma 4.4]{LuWZ} and will be used several times in later 
sections.

\begin{theorem}
\label{xxthm1.5} 
Let $A$ be a noetherian algebra such that its center $Z(A)$ is 
$\Bbbk[x]$. Let ${\mathcal P}$ be a Morita invariant property 
{\rm{(}}respectively, stable property{\rm{)}} such that the 
${\mathcal P}$-discriminant of $A$, denoted by $d$, is a 
nonzero non-invertible element in $Z(A)$. Then $A$ is strongly 
m-$Z$-retractable {\rm{(}}respectively, strongly $Z$-retractable{\rm{)}}
and strongly m-cancellative {\rm{(}}respectively, strongly 
cancellative{\rm{)}}.
\end{theorem}

\begin{proof} By \cite[Lemma 6.1]{LuWZ}, ${\mathcal P}$ is 
stable (when $\Bbbk$ is algebraically closed). The assertion 
follows from \cite[Corollary 3.11]{LuWZ} and 
\cite[Lemma 4.4]{LuWZ}.
\end{proof}

\section{Results not involving cancellation properties}
\label{xxsec2} 

In this section we collect some results that do not directly 
involve cancellation properties, but are needed in later 
sections. In the next section, we collect some lemmas that 
are directly related to cancellation properties. 

\begin{lemma}
\label{xxlem2.1} Let $\Bbbk$ be a field of characteristic zero and 
$q\neq 1$ be a nonzero scalar in $\Bbbk$. 
The following hold.
\begin{enumerate}
\item[(1)] Algebras 
$\Bbbk\langle x,y\rangle/(xy-yx)$, 
$\Bbbk\langle x,y\rangle/(xy-yx-1)$ and 
$\Bbbk\langle x,y\rangle/(xy-yx-x)$ are pairwise
not Morita equivalent.
\item[(2)]
If $q$ is not a root of unity, 
$\Bbbk\langle x,y\rangle/(xy-qyx)$ is not Morita equivalent to 
$\Bbbk\langle x,y\rangle/(xy-qyx-1)$.
\item[(3)]
The Jordan plane {\rm{(}}algebra{\rm{)}}
$\Bbbk\langle x,y\rangle/(xy-yx+x^2)$ is not Morita equivalent to 
$\Bbbk\langle x,y\rangle/(xy-yx+x^2-1)$.
\item[(4)]
The Jordan plane {\rm{(}}algebra{\rm{)}}
$\Bbbk\langle x,y\rangle/(xy-yx+x^2)$ is not Morita equivalent to 
$\Bbbk\langle x,y\rangle/(xy-yx+x^2-x)$.
\item[(5)]
The Jordan plane {\rm{(}}algebra{\rm{)}}
$\Bbbk\langle x,y\rangle/(xy-yx+x^2)$ is not Morita equivalent to 
$\Bbbk\langle x,y\rangle/(xy-yx+x^2-y)$.
\end{enumerate}
\end{lemma}

\begin{proof} (1) First of all $\Bbbk\langle x,y\rangle/(xy-yx)$
and $\Bbbk\langle x,y\rangle/(xy-yx-x)$ have global dimension two 
while $\Bbbk\langle x,y\rangle/(xy-yx-1)$ has global dimension one.
So either the algebra $\Bbbk\langle x,y\rangle/(xy-yx)$ or the 
algebra $\Bbbk\langle x,y\rangle/(xy-yx-x)$ is not Morita equivalent 
to $\Bbbk\langle x,y\rangle/(xy-yx-1)$. Second, the centers of 
$\Bbbk\langle x,y\rangle/(xy-yx)$ and $\Bbbk\langle x,y\rangle/(xy-yx-x)$
are non-isomorphic, so these algebras are not Morita equivalent.

(2) Let $A:=\Bbbk\langle x,y\rangle/(xy-qyx)$ and 
$B:=\Bbbk\langle x,y\rangle/(xy-qyx-1)$. Suppose on the contrary
that $A$ is Morita equivalent to $B$. Let $J$ be the height 
one prime ideal of $B$ generated by $(1-q)xy-1$ such that 
$B/J=\Bbbk[x^{\pm 1}]$ (with the image of $y$ being $(1-q)^{-1}x^{-1}$). 
Since $A$ is Morita 
equivalent to $B$, there is an ideal $I$ of $A$ such that
$A/I$ is Morita equivalent to $\Bbbk[x^{\pm 1}]$. Since
every projective module over $\Bbbk[x^{\pm 1}]$ is free,
$A/I$ is a matrix algebra over $\Bbbk[x^{\pm 1}]$. 
When $q$ is not a root of unity, the only height one prime 
ideals $I$ of $A$ are $(x)$ or $(y)$ \cite[Example II.1.2]{BG}. In both
cases, $A/I$ is isomorphic to $\Bbbk[t]$, which is not
a matrix algebra over $\Bbbk[x^{\pm 1}]$. This yields a 
contradiction and therefore $A$ is not Morita equivalent to 
$B$. 

(3) Let $A:=\Bbbk\langle x,y\rangle/(xy-yx+x^2)$ and 
$B:=\Bbbk\langle x,y\rangle/(xy-yx+x^2-1)$ by recycling 
notation from the proof of part (2) and suppose on the contrary
that $A$ is Morita equivalent to $B$. 
Let $J_{\pm}$ be the height one prime ideals of $B$ generated 
by $(xy-yx, x\pm 1)$. Since $A$ is Morita equivalent to $B$, 
there is an ideal $I_{\pm}$ of $A$ such that $A/I_{\pm}$ is 
Morita equivalent to $B/J_{\pm}$. Since ${\text{char}}\; \Bbbk=0$, 
$A$ has only a single height one prime that is $(x)$ \cite[Theorem 2.4]{Sh}. 
This yields a contradiction. Therefore $A$ is not Morita 
equivalent to $B$.

(4) The assertion follows from part (3) because 
$\Bbbk\langle x,y\rangle/(xy-yx+x^2-1)\cong 
\Bbbk\langle x,y\rangle/(xy-yx+x^2-x)$. 

(5) Let $A:=\Bbbk\langle x,y\rangle/(xy-yx+x^2)$ and 
$B:=\Bbbk\langle x,y\rangle/(xy-yx+x^2-y)$ by recycling 
notation from the proof of part (2) and suppose on the contrary
that $A$ is Morita equivalent to $B$. Let $y'=y-x^2$. Then 
the relation in $B$ becomes $xy'-y'x-y'=0$. Exchanging 
$x$ and $y'$, one sees that $B$ is isomorphic to 
$\Bbbk\langle x,y\rangle/(xy-yx+x)$. Let $I$ be the 
unique height one prime ideal of $B$ generated by $x$. 
Then $B/I\cong \Bbbk[y]$. Since $A$ and $B$ 
are Morita equivalent, there is a height one prime
$J$ of $A$. Since the only height one prime of $A$ is 
$(x)$. Let $J=(x)$, then $I^2$ corresponds to $J^2$. 
This implies that $B/I^2\cong \Bbbk\langle x,y\rangle/(xy-yx+x,x^2)$
is Morita equivalent to $A/J^2\cong \Bbbk[x,y]/(x^2)$. 
Since the center is preserved by Morita equivalence,
$$\Bbbk[x,y]/(x^2)\cong Z(\Bbbk\langle x,y\rangle/(xy-yx+x,x^2))
\cong \Bbbk$$
yielding a contradiction. Therefore $A$ and $B$ are 
not Morita equivalent.
\end{proof}

We thank one of the referees for the following remark.

\begin{remark}[Referee]
\label{xxrem2.2}
With only a little more work, one can show the following Claim:
Let $A$ and $B$ be two filtered AS-regular algebras of global 
dimension 2 that are not PI. Then $A$ and $B$ are Morita 
equivalent if and only if they are isomorphic.

\smallskip

\noindent
{\it Sketch proof of the claim:}
Such algebras all have the form $\Bbbk\langle x,y\rangle/(f)$
where $f$ is a (not necessarily homogeneous) polynomial of 
degree 2. By \cite[Corollary 2.12]{Ga2}, up to isomorphism, 
$f$ is one of
\begin{enumerate}
\item[(a)] 
$xy - yx - 1$,
\item[(b)] 
$xy -yx - x$,
\item[(c)] 
$xy - yx + x^2$,
\item[(d)] 
$xy - yx + x^2 + 1$,
\item[(e)] 
$xy- q yx$ for $q\neq 1$ a nonroot of unity (the algebra is denoted by
${\mathcal O}_q(\Bbbk^2)$), or
\item[(f)] 
$xy - q yx - 1$ 
for $q\neq 1$ a nonroot of unity.
\end{enumerate}

By \cite[Lemma 3.1.1]{RS}, if two noetherian domains
$A$ and $B$ are Morita equivalent, then the rings of fractions
$Q(A)$ and $Q(B)$ are isomorphic.

Note that the algebra corresponding to (a) is simple, consequently,
it is not Morita equivalent to any other algebras that are not 
simple (in cases (b)-(f)). In cases (b)-(d) the corresponding 
algebras all have $Q(A) \cong D_1(\Bbbk)$, the Weyl division ring. 
By Lemma \ref{xxlem2.1}, for the most part, these algebras 
are pairwise non-Morita equivalent. The proof of (b) not being 
Morita equivalent to (d) is similar to the proof of 
Lemma \ref{xxlem2.1}(3).

On the other hand, for cases (e) and (f), $Q(A)\cong D^q_1(\Bbbk)$,
the skew Weyl division ring. By \cite[Corollaire 3.11(c)]{AF},
$D_1^q(\Bbbk) \cong D^{q'}_1(\Bbbk)$ if and only if $q'=q^{\pm 1}$.
By Lemma \ref{xxlem2.1}(2), algebras in cases (e) and (f) are Morita 
distinct from each other. Since ${\mathcal O}_q(\Bbbk^2) \cong 
{\mathcal O}_{q^{-1}}(\Bbbk^2)$ in the case (e) (and similarly for 
the quantum Weyl algebras in case (f)), this completes the proof.
\end{remark}

Next we prove Proposition \ref{xxpro0.6}. To save some space, we refer the reader to
\cite{Le, St} for the definitions of Auslander regularity and 
Cohen-Macaulay property. A nice result of \cite[Corollary 6.2]{Le} 
is that every AS-regular algebra of global dimension three is 
Auslander regular and Cohen-Macaulay. A ring $A$ is called {\it 
stably free} if, for every finitely generated projective $A$-module
$P$, there exist integers $n$ and $m$ such that $P\oplus A^{\oplus n} 
\cong A^{\oplus m}$. Connected graded algebras are automatically 
stably free \cite{St}. An Ore domain $A$ is called a {\it maximal order} 
if $A\subseteq B$ inside the quotient division ring $Q(A)$ of 
$A$ for some ring $B$ with the property that $aBb\subseteq A$, 
for some $a,b\in A\setminus \{0\}$, then $A=B$. The main result 
of \cite{St} is

\begin{theorem} \cite[Theorem]{St}
\label{xxthm2.3}
Let $A$ be a noetherian algebra that is Auslander regular, 
Cohen-Macaulay and stably free. Then, $A$ is a domain and a maximal
order in its quotient division ring $Q(A)$.
\end{theorem}

\begin{lemma}
\label{xxlem2.4} 
Let $Z$ be a connected graded domain of GK dimension 
one. 
\begin{enumerate}
\item[(1)]
It is noetherian and finitely generated over $\Bbbk$.
\item[(2)]
If $Z$ is normal, then $Z$ is isomorphic to $\Bbbk[t]$.
\end{enumerate}
\end{lemma}

\begin{proof} Note that every domain of GK dimension one is commutative.

(1) Since $Z$ is connected graded and $\Bbbk$ is algebraically 
closed, $Z$ is a subring of $\Bbbk[t]$ where $\deg t=1$. From this, 
it is easy to see that $Z$ is finitely generated and noetherian.

(2) First of all, $\Kdim Z=\GKdim Z=1$. By part (1), $Z$ is 
noetherian. Every noetherian normal domain $Z$ of Krull 
dimension one or zero is regular (namely, has finite global 
dimension). So $Z$ is regular of global dimension no more than 
one. Since $Z$ is connected graded, its graded maximal ideal is 
principal, which implies that $Z\cong \Bbbk[t]$.
\end{proof}

Note that a noetherian commutative maximal order is a normal
domain.

\begin{lemma}
\label{xxlem2.5} 
Let $A$ be a domain that is a maximal order. 
\begin{enumerate}
\item[(1)] 
Its center $Z(A)$ is a maximal order in the field of fractions 
$Q(Z(A))$. 
\item[(2)]
If $A$ is connected graded and $\GKdim Z(A)\leq 1$, then $Z(A)$ 
is either $\Bbbk$ or $\Bbbk[t]$.
\end{enumerate}
\end{lemma}

\begin{proof}
(1) Let $B$ be a subring of $Q(Z(A))$ containing $Z(A)$ 
such that $a B b\subseteq Z(A)$ for some 
$a,b\in Z(A)\setminus \{0\}$. Let $C:=AB$ be the subring 
of $Q(A)$ generated by $A$ and $B$. Then 
$aC b\subseteq A Z(A)=A$. Since $A$ is a maximal order,
$C=A$. As a consequence, $B=Z(A)$. The assertion follows.

(2) The assertion follows from part (1) and Lemma 
\ref{xxlem2.4}.
\end{proof}

Note that if $\GKdim Z(A)=2$, then $Z(A)$ may not be regular.
For example let $A=\Bbbk_{p_{ij}}[x_1,x_2,x_3,x_4]$ where
$$p_{ij}=\begin{cases} 1 & (i,j)=(1,2),\\
-1 & (i,j)=(1,3),(2,3),(1,4),(2,4),\\
q & (i,j)=(3,4)\end{cases}$$
where $q$ is not a root of unity. Then it is easy to see
that $Z(A)$ is the second Veronese subring $\Bbbk[x_1,x_2]^{(2)}$
of the commutative polynomial ring. Hence $Z(A)$ is not regular.

\begin{lemma}
\label{xxlem2.6}
Let $A$ be a connected graded domain and $t$ be a central element 
in $A$ of positive degree $d$. 
\begin{enumerate}
\item[(1)]
For every $\alpha\in \Bbbk^{\times}$, $A/(t-\alpha)$ contains
$(A[t^{-1}])_0$ as a subalgebra.
\item[(2)]
Suppose that $A$ is generated in degree $1$ and that $d\neq 0$ in 
$\Bbbk$. Then
$$\gldim A/(t-\alpha)=\gldim (A[t^{-1}])_0.$$
\item[(3)]
Suppose that $A$ is generated in degree $1$ and that $d\neq 0$ in 
$\Bbbk$. If $A$ has finite global dimension, then so does 
$A/(t-\alpha)$ for all $\alpha\in \Bbbk^{\times}$.
\end{enumerate}
\end{lemma}

\begin{proof} (1) Let $T$ denote the $d$th Veronese subalgebra 
of $A$ where $d=\deg t$. So, in $T$, $t$ can be treated as an 
element of degree $1$. Now 
\begin{equation}
\label{E2.5.1}\tag{E2.5.1}
T/(t-\alpha)\cong T/(\alpha^{-1}t-1)\cong 
(T[(\alpha^{-1}t)^{-1}])_0= (T[t^{-1}])_0\cong (A[t^{-1}])_0
\end{equation}
where the second $\cong$ is due to \cite[Lemma 2.1]{RSS}. 

Note that $A/(t-\alpha)$ is a ${\mathbb Z}/(d)$-graded 
algebra with the degree 0 component being $T/(t-\alpha)$. 
By \eqref{E2.5.1} $A/(t-\alpha)$ contains $(A[t^{-1}])_0$ as 
a subalgebra.

(2) Since $A$ is generated in degree $1$, $A/(t-\alpha)$ is a 
strongly ${\mathbb Z}/(d)$-graded algebra with the degree 0 
component being $(A[t^{-1}])_0$. Since we assume 
$d\neq 0$ in $\Bbbk$, by \cite[Lemma 2.2(iii)]{Yi}, 
$$\gldim A/(t-\alpha)=\gldim (A[t^{-1}])_0.$$

(3) By part (2) it suffices to show that 
$(A[t^{-1}])_0$ has finite global dimension.
Since $A$ has finite global dimension, $A$ has finite graded
global dimension. Then $A[t^{-1}]$ has finite graded global
dimension. As a consequence, $(A[t^{-1}])_0$ has finite 
global dimension as required.
\end{proof}

To conclude this section we list two well-known results.

\begin{lemma} \cite[Lemma 7.6]{LPWZ}
\label{xxlem2.7}
Let $A$ be a connected graded algebra and $t$ be a central 
element of degree $1$. If $A/(t)$ has finite global 
dimension, then so does $A$.
\end{lemma}

\begin{lemma} \cite[Corollary 2]{SmZ}
\label{xxlem2.8}
Let $A$ be a finitely generated Ore domain that is not PI. 
Let $Z$ be the center of $A$. Then $\GKdim Z\leq \GKdim A-2.$
\end{lemma}

\section{Some cancellation lemmas}
\label{xxsec3}

First we recall a classical result concerning the 
cancellation property.

\begin{lemma} \cite[Corollary 3.4]{AEH}
\label{xxlem3.1} 
Let $A$ be an affine domain of GK dimension at most one.
\begin{enumerate}
\item[(1)]
If $A=\Bbbk$, then it is trivially strongly
retractable and strongly cancellative.
\item[(2)]
If $A=\Bbbk[t]$, then it is strongly cancellative.
\item[(3)]
If $A\not\cong \Bbbk[t]$, then it is strongly
retractable, and consequently, strongly cancellative.
\end{enumerate}
\end{lemma}

The following lemma concerns cancellation properties for
a tensor product $A\otimes R$ where $R$ is commutative. 

\begin{lemma}
\label{xxlem3.2}
Let $A$ be an algebra with trivial center and let $R$ be a commutative
algebra that is cancellative {\rm{(}}respectively, strongly 
cancellative{\rm{)}}. Then the tensor product $A\otimes R$ is both 
cancellative {\rm{(}}respectively, strongly cancellative{\rm{)}} and 
m-cancellative {\rm{(}}respectively, strongly m-cancellative{\rm{)}}.
\end{lemma}

\begin{proof} The proofs for the assertions without the word ``strongly'' 
are similar by taking $n=1$ in the following proof. So we only prove 
the ``strongly'' version.

First we show that $A\otimes R$ is strongly cancellative assuming
that $R$ is strongly cancellative. Let $B$ be an algebra such that
$$\phi: (A\otimes R)[t_1,\cdots,t_n]\xrightarrow{\cong} B[s_1,\cdots,s_n]$$
is an isomorphism of algebras. Taking the center on both sides, 
we obtain an isomorphism
$$\phi_Z: R[t_1,\cdots,t_n]\xrightarrow{\cong} Z(B)[s_1,\cdots,s_n]$$
where $\phi_Z$ is a restriction of $\phi$ on the centers. 
Since $R$ is strongly cancellative, $R\cong Z(B)$. 
%Let $$f_i=\phi_Z^{-1}(s_i)=\phi^{-1}(s_i)\in R[t_1,\cdots,t_n]$$
%for $i=1,\cdots,n$. 
Let $I$ be the ideal of $Z(B)[s_1,\cdots,s_n]$
generated by $\{s_i\}_{i=1}^n$. Then $J:=\phi^{-1}(I)$ is an ideal of
$R[t_1,\cdots,t_n]$ and 
$$A\otimes (R[t_1,\cdots,t_n]/J)\cong B\otimes (\Bbbk[s_1,\cdots,s_n]/I)
\cong B.$$
Taking the center on both sides of the above isomorphism and using the 
fact that $Z(A)=\Bbbk$, we have
$$R[t_1,\cdots,t_n]/J\cong Z(B)\cong R.$$
Therefore 
$$B\cong A\otimes (R[t_1,\cdots,t_n]/J)\cong A\otimes R$$
as required.

Next we show that if $R$ is strongly cancellative, then $A\otimes R$ 
is strongly m-cancellative. Let $B$ be an algebra such that
$$A':=(A\otimes R)[t_1,\cdots,t_n] \quad
{\text{is Morita equivalent to}} \quad B[s_1,\cdots,s_n]=:B'.$$
By \cite[Lemma 3.1(iii)]{LuWZ}, there is an invertible $(A',B')$-bimodule
$\Omega$ and an isomorphism 
$$\omega: Z(A')=R[t_1,\cdots,t_n]\xrightarrow{\cong}
Z(B)[s_1,\cdots,s_n]=Z(B')$$
such that the left action of $x\in Z(A')$ on $\Omega$ agrees with the 
right action of $\omega(x)\in Z(B')$ on $\Omega$. Since $R$ is strongly 
cancellative, $R\cong Z(B)$. 
%Let $$f_i=\omega^{-1}(s_i)\in R[t_1,\cdots,t_n]$$
%for $i=1,\cdots,n$. 
Let $I$ be the ideal of $Z(B)[s_1,\cdots,s_n]$
generated by $\{s_i\}_{i=1}^n$. Then $J:=\omega^{-1}(I)$ is an ideal 
of $R[t_1,\cdots,t_n]$, and by \cite[Lemma 3.1(v)]{LuWZ}, 
$$A\otimes (R[t_1,\cdots,t_n]/J) \quad
{\text{is Morita equivalent to}} \quad B\otimes (\Bbbk[s_1,\cdots,s_n]/I)
\cong B.$$
Taking the center of the above Morita equivalence and using the fact
that $Z(A)=\Bbbk$, we have
$$R[t_1,\cdots,t_n]/J\cong Z(B)\cong R.$$
Hence
$$A\otimes (R[t_1,\cdots,t_n]/J)\cong A\otimes R.$$
Therefore $A\otimes R$ is Morita equivalent to $B$
as required.
\end{proof}

\begin{corollary}
\label{xxcor3.3}
Let $\Bbbk$ be of characteristic zero and $A$ be a commutative algebra.
Let $\delta$ be a locally nilpotent derivation of $A$ with $\delta(y)=1$ 
for some $y\in A$. Suppose that $\ker(\delta)$ is cancellative 
{\rm{(}}respectively, strongly cancellative{\rm{)}}. Then $A[x;\delta]$ 
is cancellative {\rm{(}}respectively, strongly cancellative{\rm{)}}.
\end{corollary}

\begin{proof}
Let $C=\ker (\delta)$. By \cite[Lemma 14.6.4]{MR} $A[x;\delta]\cong 
C\otimes A_1(\Bbbk)$. By hypothesis, then $C$ is cancellative and 
$Z(A_1(\Bbbk))=\Bbbk$. The assertion follows from Lemma \ref{xxlem3.2}.
\end{proof}

With a slight modification to the previous lemma we can consider the 
case in which $R$ is a (noncommutative) $Z$-retractable algebra.

\begin{lemma}
\label{xxlem3.4}
Let $A$ be an algebra with trivial center and let $R$ be a 
$Z$-retractable algebra {\rm{(}}respectively, strongly 
$Z$-rectractable{\rm{)}}. Then the tensor product $A\otimes R$ 
is $Z$-rectractable {\rm{(}}respectively, strongly $Z$-retractable{\rm{)}}.
\end{lemma}

The proof of Lemma \ref{xxlem3.4} is similar to the proof of Lemma 
\ref{xxlem3.2}, so it is omitted. We call an algebra $R$ 
{\it universally right noetherian} if $A\otimes R$ is right 
noetherian for every right noetherian $\Bbbk-$algebra $A$. This 
property was studied in \cite{ASZ}. Thus, if in the previous 
lemma, we suppose that $A$ is right noetherian algebra and $R$ is 
universally right noetherian, then $A\otimes R$ is Hopfian, and 
by \cite[Lemma 3.6]{LeWZ} $A\otimes R$ is cancellative (respectively, 
strongly cancellative).

The following lemma is useful in some situations.

\begin{lemma}
\label{xxlem3.5}
Let $A$ be a noetherian algebra such that
\begin{enumerate}
\item[(i)]
its center $Z(A)$ is the commutative polynomial ring 
$\Bbbk[t]$ for some $t\in A$, and 
\item[(ii)]
$t$ is in the ideal $[A,A]$ of $A$ generated by the commutators
and $[A,A]\neq A$.
\end{enumerate}
Then $A$ is strongly $Z$-retractable and strongly cancellative.
\end{lemma}

\begin{proof} Let ${\mathcal P}$ be the property that the 
commutators generate the whole algebra. By 
(ii) the property ${\mathcal P}$ fails for the maximal
ideal $(t)$ in $\Bbbk[t]$ since $[\overline{A},\overline{A}]
\neq \overline{A}$ where $\overline{A}=A/(t)$. By (ii) again, 
$(t-\alpha)+[A,A]=A$ for all $\alpha\neq 0$. This means that
the property ${\mathcal P}$ holds for the maximal ideal 
$(t-\alpha)$ in $\Bbbk[t]$ for all $\alpha\neq 0$. Thus the 
${\mathcal P}$-discriminant of $A$ is $t$. The assertion 
follows from Theorem \ref{xxthm1.5}.
\end{proof}

We refer to \cite{LeWZ, LuWZ} for the definition of 
${\text{LND}}^H_t$-rigid in the proof of the following
lemma.

\begin{lemma}
\label{xxlem3.6}
Let $A$ be a noetherian domain with $Z(A)=\Bbbk[t]$.
Let ${\mathcal P}$ be a property such that the 
${\mathcal P}$-discriminant is $t$. 
\begin{enumerate}
\item[(1)]
$A$ is strongly $Z$-retractable and 
strongly cancellative.
\item[(2)]
If ${\mathcal P}$ is a Morita invariant, then 
$A$ is strongly $Z$-retractable, strongly 
m-$Z$-retractable, strongly cancellative and strongly 
m-cancellative.
\end{enumerate}
\end{lemma}

\begin{proof} We only prove (2). 

Since $t$ is an effective element in $\Bbbk[t]$ by 
\cite[Example 2.8]{LeWZ}. By \cite[Theorem 3.10]{LuWZ},
$Z$ is strongly ${\text{LND}}^H_t$-rigid. By 
\cite[Proposition 3.7(ii)]{LuWZ}, $A$ is both strongly 
$Z$-retractable and strongly m-$Z$-retractable. Since 
$A$ is noetherian, it is Hopfian in the sense of 
\cite[Definition 3.4]{LeWZ}. By \cite[Lemmas 4.4 and 4.6]{LuWZ}, 
$A$ is strongly cancellative and strongly m-cancellative.
\end{proof}

Next we consider the connected graded case.

\begin{lemma}
\label{xxlem3.7} 
Let $A$ be a noetherian connected graded domain.
\begin{enumerate}
\item[(1)]
If $Z(A)$ has GK dimension $\leq 1$ and $Z(A)$ is not isomorphic to 
$\Bbbk[t]$, then 
$A$ is strongly $Z$-retractable, strongly m-$Z$-retractable, 
strongly cancellative and strongly m-cancellative.
\end{enumerate}
For the following parts, we assume that $A$ is generated in degree $1$,
that $Z(A)\cong \Bbbk[t]$ and that ${\text{char}}\; \Bbbk=0$.
\begin{enumerate}
\item[(2)]
If $\gldim A/(t)=\infty$ and $\gldim A/(t-1)<\infty$, then 
$A$ is strongly $Z$-retractable, strongly m-$Z$-retractable, 
strongly cancellative and strongly m-cancellative.
\item[(3)]
Suppose the global dimension of $A$ is finite and 
$\gldim A/(t)=\infty$. Then $A$ is strongly $Z$-retractable, 
strongly m-$Z$-retractable, strongly cancellative and 
strongly m-cancellative.
\item[(4)]
Suppose $A$ is AS-regular and $\gldim A/(t)=\infty$. 
Then $A$ is strongly $Z$-retractable, 
strongly m-$Z$-retractable, strongly cancellative and 
strongly m-cancellative.
\end{enumerate}
\end{lemma}

\begin{proof} (1) By Lemma \ref{xxlem2.4}(1), $Z$ is an affine
domain. By Lemma \ref{xxlem3.1}, $Z$ is strongly 
retractable. By taking ${\mathcal P}$ to be a trivial property, 
say being an algebra, the ${\mathcal P}$-discriminant is 1. 
By \cite[Remark 3.7(6)]{LeWZ}, $Z$ is strongly ${\text{LND}}^H_1$-rigid.
By \cite[Proposition 3.7(ii)]{LuWZ}, $A$ is both strongly $Z$-retractable 
and strongly m-$Z$-retractable. Since $A$ is noetherian, it is 
Hopfian in the sense of \cite[Definition 3.4]{LeWZ}. By 
\cite[Lemmas 4.4 and 4.6]{LuWZ}, 
$A$ is strongly cancellative and strongly m-cancellative.

(2) Let ${\mathcal P}$ be the property of having finite 
global dimension. By Lemma \ref{xxlem2.6}, for every
$0\neq \alpha\in \Bbbk$,
$$\gldim A/(t-\alpha)=\gldim (A[t^{-1}])_{0}
=\gldim A/(t-1)<\infty$$
and by the hypothesis, we have that
$$\gldim A/(t)=\infty.$$
Hence the ${\mathcal P}$-discriminant is $t$. The assertion 
follows from Lemma \ref{xxlem3.6}(2). 

(3) By Part (2), it suffices to show that $\gldim A/(t-1)<\infty$.
Since $A$ has finite global dimension, so does $A[t^{-1}]$. 
Then $A[t^{-1}]$ has finite graded global dimension. Since $A$ 
is generated in degree 1, $A[t^{-1}]$ is strongly ${\mathbb Z}$-graded.
Hence $\gldim (A[t^{-1}])_{0}$ is finite. By Lemma \ref{xxlem2.6}, 
$$\gldim A/(t-1)=\gldim (A[t^{-1}])_{0}<\infty$$
as required.

(4) The assertion follows from part (3) and the fact that
an AS-regular algebra has finite global dimension.
\end{proof}

By Lemma \ref{xxlem3.7}(1), the case of $\GKdim Z(A)=1$ is covered
except for $Z(A)=\Bbbk[t]$.

\begin{lemma}
\label{xxlem3.8}
Let $A$ be a noetherian connected graded algebra.
\begin{enumerate}
\item[(1)]
Suppose $Z(A)=\Bbbk[t]$ for some homogeneous element $t$ of positive 
degree. If $(A[t^{-1}])_0$ does not have any nonzero 
finite dimensional left module, then $A$ is strongly $Z$-retractable, 
strongly m-$Z$-retractable, strongly cancellative and strongly 
m-cancellative.
\item[(2)]
Suppose $B$ is a connected graded subalgebra of $A$ satisfying 
\begin{enumerate}
\item[(i)]
$Z(B)=\Bbbk[t]$ for some homogeneous element $t\in B$ of positive 
degree.
\item[(ii)]
$Z(A)\cap Z(B)\neq \Bbbk$.
\item[(iii)]
$(A[t^{-d}])_0$ does not have any nonzero 
finite dimensional left module for some $t^d\in Z(B)\cap Z(A)$,
where $d$ is a positive integer.
\item[(iv)]
$A_B$ is finitely generated and $B$ is noetherian.
\item[(v)]
$A=B\oplus C$ as a right $B$-module. 
\end{enumerate}
Then $B$ is strongly $Z$-retractable, 
strongly m-$Z$-retractable, strongly cancellative and strongly 
m-cancellative.
\end{enumerate}
\end{lemma}

\begin{proof} (1)
Let ${\mathcal P}$ be the property of not having nonzero
finite dimensional left module over an algebra. Since $A$ is connected
graded, ${\mathcal P}$ fails for $A/(t)$. We claim that 
${\mathcal P}$ holds for $A/(t-\alpha)$ for all $\alpha\in \Bbbk^{\times}$.
By the hypothesis, $(A[t^{-1}])_0$ does not have any nonzero 
finite dimensional left module. Since $A/(t-\alpha)$ contains
$(A[t^{-1}])_0$ by Lemma \ref{xxlem2.6}(1), $A/(t-\alpha)$ does 
not have any nonzero finite dimensional left module. So the claim 
holds. Therefore the ${\mathcal P}$-discriminant of
$A$ is $t$. Now the assertion follows from Lemma \ref{xxlem3.6}(2). 

(2) By part (1) it suffices to show that $(B[t^{-1}])_0$ does not 
have any nonzero finite dimensional left module. Note that 
$B[t^{-1}]=B[t^{-d}]$. So it is equivalent to show that 
$(B[t^{-d}])_0$ does not have any nonzero finite dimensional 
left module. We prove this claim by contradiction. Suppose 
%otherwise
$M$ is a nonzero finite dimensional left 
$(B[t^{-d}])_0$-module. By hypotheses (iii)-(iv) and by inverting 
$t^{d}$, $A[t^{-d}]=C[t^{-d}]\oplus B[t^{-d}]$ and $A[t^{-d}]$ is a
finitely generated right $B[t^{-d}]$-module.
Then 
$(A[t^{-d}])_0=(C[t^{-d}])_0\oplus (B[t^{-d}])_0$ and $(A[t^{-d}])_0$ 
is a finitely generated right $(B[t^{-d}])_0$-module. Hence
$(A[t^{-d}])_0\otimes_{(B[t^{-d}])_0} M$ is a nonzero finite 
dimensional left $(A[t^{-d}])_0$-module. This yields a contradiction.
At this point, we have proved that $(B[t^{-1}])_0$ does not 
have any nonzero finite dimensional left module. The assertion follows
from part (1).
\end{proof}

Lemma \ref{xxlem3.8} can be applied to many examples. Here is an 
easy one.

\begin{example}
\label{xxex3.9} 
Suppose ${\text{char}}\; \Bbbk=0$. Let $A$ be a generic 3-dimensional 
Sklyanin algebra generated by $\{x,y,z\}$, see \cite[Introduction]{GKMW} 
for the relations. Then $Z(A)=\Bbbk[g]$ where $g$ is a homogeneous 
element of degree three. Let $G$ be any finite group of graded algebra 
automorphisms of $A$. Let $B$ be the fixed subring $A^G$. Then we claim
that $B$ is strongly $Z$-retractable, strongly m-$Z$-retractable, 
strongly cancellative and strongly m-cancellative. It is easy to see 
that hypotheses (ii), (iii), (iv) and (v) in Lemma \ref{xxlem3.8}(2) 
hold. 
%If hypothesis (i) in Lemma \ref{xxlem3.8}(2) fails, then the 
%claim follows by Lemma \ref{xxlem3.7}(1). If hypothesis (i) in Lemma 
%\ref{xxlem3.8}(2) holds, then the claim follows by Lemma 
%\ref{xxlem3.8}(2).
If hypothesis (i) in Lemma \ref{xxlem3.8}(2) holds, then
it is easy to see that remaining hypotheses hold as well 
and the claim follows by that Lemma. If hypothesis (i) in
Lemma \ref{xxlem3.8}(2) fails, then the result follows by Lemma 
\ref{xxlem3.7}(1).
\end{example}

To conclude this section we give an example of ungraded algebras that 
are cancellative.

\begin{example}
\label{xxex3.10}
Suppose ${\text{char}}\; \Bbbk=0$.
Let $A$ be the universal enveloping algebra $U({\mathfrak g})$
where ${\mathfrak g}$ is a $3$-dimensional non-abelian Lie algebra.
One can use Bianchi classification to list all $3$-dimensional 
non-abelian Lie algebras \cite[Section 1.4]{Ja} as follows. 
\begin{enumerate}
\item[(1)]
${\mathfrak g}=sl_2$.
\item[(2)]
${\mathfrak g}$ is the Heisenberg Lie algebra
\item[(3)]
${\mathfrak g}=L\oplus \Bbbk z$ where $L$ is the $2$-dimensional
non-abelian Lie algebra.
\item[(4)] 
${\mathfrak g}$ has a basis $\{e,f,g\}$ and subject to the following 
relations
\begin{align*}
[e,f]&=0,& [e,g]&=e,& [f,g]&=\alpha f
\end{align*}
where $\alpha\neq 0$.
\item[(5)] 
${\mathfrak g}$ has a basis $\{e,f,g\}$ and subject to the following 
relations
\begin{align*}
[e,f]&=0,& [e,g]&=e+\beta f,& [f,g]&= f
\end{align*}
where $\beta\neq 0$.
\end{enumerate}

For each class, one can verify that $A$ is strongly cancellative. 

(1) See \cite[Example 6.11]{LuWZ}. 

(2) The universal enveloping algebra of the Heisenberg Lie algebra has the 
center $Z=\Bbbk[t]$ with $t$ in the ideal generated by the commutators, 
then the assertion follows from Lemma \ref{xxlem3.5}. 

(3) In this case $U({\mathfrak g})=U(L)\otimes \Bbbk [z]$ with $Z(U(L))=
\Bbbk$. The assertion follows from Lemma \ref{xxlem3.2}.

(4,5) In both cases, one can write $A:=U({\mathfrak g})$ as an Ore extension 
$\Bbbk[e,f][g;\delta]$ for some derivation $\delta$ of the commutative 
polynomial ring $\Bbbk[e,f]$. If the center of $A$ is trivial, then
$A$ is strongly cancellative by \cite[Proposition 1.3]{BZ1}. For the 
rest of proof we assume that $Z(A)\neq \Bbbk$. Note that the derivation
$\delta$ of $\Bbbk[e,f]$ is determined by
\begin{equation}
\label{E3.10.1}\tag{E3.10.1}
\delta: \quad e\longrightarrow -e, \;\; f\longrightarrow -\alpha f
\end{equation}
in part (4), and by 
\begin{equation}
\label{E3.10.2}\tag{E3.10.2}
\delta: \quad e\longrightarrow -(e+\beta f), \;\; f\longrightarrow -f
\end{equation}
in part (5). By an easy calculation, one sees that
$$Z(A)=\{ x\in \Bbbk[e,f]\mid \delta(x)=0\}$$
which is a graded subring of $\Bbbk[e,f]$ (which is inside $A$). Since $A$ 
contains $U(L)$ as a subalgebra where $L$ is the $2$-dimensional non-abelian 
Lie algebra, $A$ is not PI. By Lemma \ref{xxlem2.8}, $\GKdim Z(A)\leq 1$.
Since $Z(A)\neq \Bbbk$ and $\Bbbk$ is algebraically closed, $\GKdim Z(A)
\geq 1$. Thus $\GKdim Z(A)=1$. By Lemma \ref{xxlem2.4}, $Z(A)$ is a domain 
that is finitely generated over $\Bbbk$. If $Z(A)$ is not isomorphic to 
$\Bbbk[t]$, then $Z(A)$ is strongly retractable by Lemma \ref{xxlem3.1}(3). 
%(The next few lines are copied from the proof of Lemma \ref{xxlem3.7}(1).)
Using the proof of Lemma \ref{xxlem3.7}(1), one sees that $A$ is strongly 
cancellative and strongly m-cancellative. 
%By taking ${\mathcal P}$ to be a trivial property, say being an algebra, 
%the ${\mathcal P}$-discriminant is 1. By \cite[Remark 3.7(6)]{LeWZ}, $Z$ 
%is strongly ${\text{LND}}^H_1$-rigid. By \cite[Proposition 3.7(ii)]{LuWZ}, 
%$A$ is both strongly $Z$-retractable and strongly m-$Z$-retractable. Since 
%$A$ is noetherian, it is Hopfian in the sense of \cite[Definition 3.4]{LeWZ}. 
%By \cite[Lemmas 4.4 and 4.6]{LuWZ}, $A$ is strongly cancellative and 
%strongly m-cancellative. 
For the rest, we assume that $Z(A)\cong \Bbbk[t]$ 
for some homogeneous element $t$ in $\Bbbk[e,f]$. By the form of $\delta$ 
in \eqref{E3.10.1}-\eqref{E3.10.2}, the degree of $t$ is at least $2$. This 
implies that $\Bbbk[e,f]/(t)$ has infinite global dimension. On the other 
hand, if $\alpha\neq 0$, then $\Bbbk[e,f]/(t-\alpha)$ has finite global 
dimension (applying Lemma \ref{xxlem2.6} to the algebra $\Bbbk[e,f]$). 
Therefore $A/(t)=(\Bbbk[e,f]/(t))[g;\delta]$ has infinite dimensional 
dimension and $A/(t-\alpha)=(\Bbbk[e,f]/(t-\alpha))[g;\delta]$ has finite 
global dimension. Then the argument similar to the proof of
Lemma \ref{xxlem3.7}(2) shows that $A$ is strongly cancellative.
%By the way the center of $U({\mathfrak g})$ can explicitly be worked out, 
%see \cite[Example 14.4.2]{MR} for some hints. 
\end{example}

One obvious question after Example \ref{xxex3.10} is whether or not 
every universal enveloping algebra of a $4$-dimensional non-abelian 
Lie algebra is cancellative.

\section{Proof of Theorems \ref{xxthm0.2}, \ref{xxthm0.3} and 
Corollary \ref{xxcor0.5}}
\label{xxsec4}

In this section we prove some of the main results listed in 
the introduction. We start with Theorem \ref{xxthm0.3}.

\begin{proof}[Proof of Theorem \ref{xxthm0.3}]
If $Z(A)$ is not isomorphic to $\Bbbk[t]$, the assertion follows from 
Lemma \ref{xxlem3.7}(1). If $Z(A)$ is isomorphic to $\Bbbk[t]$, 
the assertion follows from Lemma \ref{xxlem3.7}(3).
\end{proof}

\begin{proof}[Proof of Theorem \ref{xxthm0.2}]
Note that every AS-regular algebra has finite global dimension. If 
$A$ is not PI, then by Lemma \ref{xxlem2.8}, 
$$\GKdim Z \leq \GKdim A-2=3-2=1.$$
If $Z(A)$ is not isomorphic to $\Bbbk[t]$, the assertion follows from 
Lemma \ref{xxlem3.7}(1). If $Z(A)$ is isomorphic to $\Bbbk[t]$ and if 
$\gldim A/(t)=\infty$, the assertion follows from Lemma 
\ref{xxlem3.7}(3). For the rest of the proof we assume that
$Z(A)=\Bbbk[t]$ and $A/(t)$ has finite global dimension.

By Rees Lemma, $\gldim A/(t)\leq 2$. By a Hilbert series computation, 
we obtain that $\GKdim A/(t)=2$. This implies that $A/(t)$ is AS-regular 
of global dimension two. Since we assume that $\Bbbk$ is algebraically 
closed, $A/(t)$ is either $\Bbbk_q[x,y]$ or $\Bbbk_{J}[x,y]$. In particular, 
the Hilbert series of $A/(t)$ is $\frac{1}{(1-s)^2}$. Since $A$ is 
AS-regular of global dimension three, it is generated by either $3$ elements or
$2$ elements. Next we consider these two cases.

Case 1: $A$ is generated by two elements. Then the Hilbert series of
$A$ is $\frac{1}{(1-s)^2(1-s^2)}$. This forces that $\deg t=2$. If $A/(t)=
\Bbbk_q[x,y]$, then $t=xy-qyx$ and $A/(t-1)=\Bbbk\langle x,y \rangle/(xy-qyx-1)$.
If $q=1$, let ${\mathcal P}$ be the property of {\bf not} being Morita 
equivalent to $A/(t)$. Then the ${\mathcal P}$-discriminant is $t$ by Lemma 
\ref{xxlem2.1}(1). Now the assertion 
follows from Lemma \ref{xxlem3.6}(2).

If $q\neq 1$, we claim that $q$ is not a root of unity. If $q$ is a 
root of unity, then $A/(t-1)$ is PI. By Lemma \ref{xxlem2.6}(1),
$(A[t^{-1}])_0$ is PI. Note that $A^{(2)}[t^{-1}]=(A[t^{-1}])_0 [t^{\pm 1}]$.
So $A^{(2)}[t^{-1}]$ is PI. Consequently, $A^{(2)}$ is PI and 
whence $A$ is PI, a contradiction. Then by the argument in the previous paragraph with 
Lemma \ref{xxlem2.1}(2) being replaced by Lemma \ref{xxlem2.1}(1), one
sees that $A$ is strongly cancellative and strongly m-cancellative.

If $A/(t)=\Bbbk_{J}[x,y]$, then $t=xy-yx+x^2$ and 
$A/(t-1)=\Bbbk\langle x,y \rangle/(xy-yx+x^2-1)$.
Let ${\mathcal P}$ be the property of being Morita equivalent to 
$A/(t-1)$. Then the ${\mathcal P}$-discriminant is $t$ by Lemma 
\ref{xxlem2.1}(3). By Lemma \ref{xxlem3.6}(2), 
$A$ is strongly cancellative and strongly m-cancellative.

Case 2: $A$ is generated by three elements. Then the Hilbert series of
$A$ is $\frac{1}{(1-s)^3}$. If $A$ is isomorphic to 
$A'\otimes \Bbbk[t]$ for some algebra $A'$, then $Z(A')$ is trivial
and the assertion follows from Lemma \ref{xxlem3.2}. So we further
assume that $A$ is not a tensor product of two nontrivial algebras.
In this case $A$ is generated by $x,y,t$ subject to three relations
$$\begin{aligned}
xt-tx&=0,\\
yt-ty&=0,
\end{aligned}
$$
\begin{equation}
\label{E4.0.1}\tag{E4.0.1}
xy-qyx= ft+\epsilon t^2, \quad 
{\text{or}}\quad 
xy-yx+x^2=ft+\epsilon t^2
\end{equation}
where $\epsilon$ is either $0$ or $1$ and $f$ is a linear combination
of $x$ and $y$. 

Again we have three cases. If $q=1$ in the first type of \eqref{E4.0.1}, 
using Lemmas \ref{xxlem2.1}(1) and \ref{xxlem3.6}(2), one sees that
$A$ is strongly cancellative and strongly m-cancellative. If $q\neq 1$
in the first type of \eqref{E4.0.1},
then we can assume that $f=0$ and $\epsilon=1$ after a base change. 
Since $A$ is not PI, $q$ is not a root of unity. Then we use 
Lemma \ref{xxlem2.1}(2) instead of Lemma \ref{xxlem2.1}(1).
Otherwise we have the relation
$$xy-yx+x^2-ft-\epsilon t^2=0.$$
Up to a base change, we may assume that $\epsilon=0$. Then we have 
either $xy-yx+x^2-xt=0$ or $xy-yx+x^2-yt=0$. In the case of
$xy-yx+x^2-xt=0$, using Lemmas \ref{xxlem2.1}(4)
and \ref{xxlem3.6}(2), one sees that
$A$ is strongly cancellative and strongly m-cancellative. 
In the case of
$xy-yx+x^2-yt=0$, using Lemmas \ref{xxlem2.1}(5)
and \ref{xxlem3.6}(2), one sees that
$A$ is strongly cancellative and strongly m-cancellative. 
\end{proof}

For the rest of this section we study cancellation property for 
some graded isolated singularities. In noncommutative algebraic 
geometry, Ueyama gave the following definition of a graded 
isolated singularity.

\begin{definition} \cite[Definition 2.2]{Ue}
\label{xxdef4.1}
Let $A$ be a noetherian connected graded algebra.
Then $A$ is called a {\it graded isolated singularity} if
\begin{enumerate}
\item[(1)]
$\gldim A$ is infinite.
\item[(2)] 
The associated noncommutative projective scheme ${\text{Proj}}\; A$ 
(in the sense of \cite{AZ}) has finite global dimension. 
\end{enumerate}
\end{definition}

Examples of graded isolated singularities are given in 
\cite{CYZ3, GKMW, MU1, MU2, Ue}. One nice example of graded 
isolated singularities is the fixed subring of the generic 
Sklyanin algebra under the cyclic permutation action 
\cite[Theorem 5.2]{GKMW} which is cancellative by Example 
\ref{xxex3.9}. 

\begin{lemma}
\label{xxlem4.2} 
Suppose ${\text{char}}\; \Bbbk=0$.
Let $A$ be a graded isolated singularity generated in degree one
and $t\in B$ be a central regular element of positive degree. Then 
$A/(t-\alpha)$, for every $0\neq \alpha\in \Bbbk$, has 
finite global dimension.
\end{lemma}

\begin{proof} By Lemma \ref{xxlem2.6}(2) it suffices to show 
that $(A[t^{-1}])_0$ has finite global dimension. Since
there is a localizing functor from ${\text{Proj}}\; A$ 
to ${\text{GrMod}}\; A[t^{-1}]$, 
$$\gldim {\text{GrMod}}\; A[t^{-1}]\leq \gldim {\text{Proj}}\; A<\infty.$$
It is well-known that
$$\gldim (A[t^{-1}])_0=\gldim {\text{GrMod}}\; A[t^{-1}]$$
as $A[t^{-1}]$ is strongly ${\mathbb Z}$-graded. The assertion
follows.
\end{proof}

\begin{theorem}
\label{xxthm4.3}
Let $A$ be a noetherian connected graded domain generated in 
degree $1$. Suppose
\begin{enumerate}
\item[(1)]
${\text{char}}\; \Bbbk=0$, and
\item[(2)]
$\GKdim Z(A)\leq 1$, and
\item[(3)]
$A$ is a graded isolated singularity.
\end{enumerate}
Then $A$ is strongly cancellative and strongly m-cancellative.
\end{theorem}

\begin{proof} 
If $Z(A)\not \cong \Bbbk[t]$, then the assertion follows from Lemma
\ref{xxlem3.7}(1). Now we assume that $Z(A)\cong \Bbbk[t]$ 
where $t$ can be chosen to be a homogeneous element of positive 
degree. Since $A$ has infinite global dimension, so does $A/(t)$ by
Lemma \ref{xxlem2.7}. For every $0\neq \alpha\in \Bbbk$,
by Lemma \ref{xxlem4.2}, $A/(t-\alpha)$ has finite global dimension.
The assertion follows from Lemma \ref{xxlem3.7}(2).
\end{proof}

Now we are ready to prove Corollary \ref{xxcor0.5}.

\begin{proof}[Proof of Corollary \ref{xxcor0.5}]
If $d=1$, then it follows from Theorem \ref{xxthm0.2}. 

Next we assume that $d>1$. Note that the Hilbert series of $A$ is either 
$\frac{1}{(1-s)^3}$ or $\frac{1}{(1-s)^2(1-s^2)}$. By an easy computation,
the Hilbert series of $A^{(d)}$ can not be of the form 
$\frac{1}{f(s)}$ for some polynomial $f(s)$. By \cite[Theorem 2.4]{StZ} 
and the argument before it, $A^{(d)}$ does not have finite global dimension. 
By \cite[Proposition 5.10(3)]{AZ}, $A^{(d)}$ is a graded isolated 
singularity. Since $A$ is not PI, $\GKdim Z(A)\leq 1$ by Lemma 
\ref{xxlem2.8}. The assertion follows from Theorem \ref{xxthm4.3}.
\end{proof}

\subsection*{Acknowledgments}
The authors thank referees for their careful reading and
valuable comments and for allowing them to include Remark 
\ref{xxrem2.2} in this paper.
X. Tang thanks J.J. Zhang for the hospitality during his short
visit to the University of Washington; and some travel support from 
Bill \& Melinda Gates Foundation is gratefully acknowledged. 
H. Venegas Ram{\' i}rez is very grateful to J.J. Zhang and the 
Department of Mathematics at the University of Washington, for 
their hospitality, advice and discussion about the subject during 
his internship. H. Venegas Ram{\' i}rez was partially supported 
by Colciencias (doctoral scholarship 757). J.J. Zhang was 
partially supported by the US National Science Foundation 
(Nos. DMS-1700825 and DMS-2001015).

%\section{Questions}
%\label{xxsec20}

%Q1: Is $\Bbbk_{q}[x^{\pm 1}, y^{\pm 1}, z^{\pm 1}][t]$ cancellative?

%Q2: Let $H$ be an affine Hopf domain (or prime) of GK dimension 2.
%Is then $H$ cancellative?

%Q3: If ${\text{char}}\; \Bbbk>0$, is then $\Bbbk_{J}[x,y]$ cancellative?

\providecommand{\bysame}{\leavevmode\hbox to3em{\hrulefill}\thinspace}
\providecommand{\MR}{\relax\ifhmode\unskip\space\fi MR }
\providecommand{\MRhref}[2]{%

\href{http://www.ams.org/mathscinet-getitem?mr=#1}{#2} }
\providecommand{\href}[2]{#2}

\end{document}